\documentclass[10pt]{amsart}
\usepackage[spanish, english]{babel}
\usepackage{amsmath,amssymb,latexsym,amsthm}
\usepackage{enumerate, mathrsfs}
\usepackage[normalem]{ulem}
\usepackage{epsfig}
\usepackage{graphicx}
\usepackage{wrapfig}
\usepackage{color}
\usepackage{tikz}

\usepackage{hyperref}

\usepackage{subcaption}

\newcommand{\R}{\mathbb{R}}

\newcommand{\so}{\Rightarrow}
\newcommand{\e}{\varepsilon}
\newcommand{\s}{\sigma}

\newcommand{\inv}{^{-1}}
\newcommand{\fe}{\varphi}

\newcommand{\del}{\partial}

\newtheorem{thm}{Theorem}[section]
\newtheorem{lem}[thm]{Lemma}

\newtheorem{prop}[thm]{Proposition}
\newtheorem{coro}[thm]{Corollary}
\theoremstyle{definition}

\newtheorem*{teon}{Theorem}

\newtheorem*{corn}{Corollary}

\newtheorem{dfn}{Definition}

\newtheorem{rem}{Remark}

\title{Bi-contact structures with symmetry: local normal forms}

\author[Connor Jackman]{Connor Jackman}
\address[Jackman]{Heidelberg, Germany}
\email{cjackman@mathi.uni-heidelberg.de}

\begin{document}
\date{\today}
\maketitle
\begin{abstract}
 A pair of transverse contact distributions on a 3-manifold will in general admit no 1-parameter families of symmetries: a flow preserving both contact distributions. Here, we will determine local normal forms for such pairs admitting symmetries. In particular, we observe that orientable Anosov flows may be globally given by the intersection of a pair of oppositely oriented contact distributions admitting, around any point, maximal local symmetries.
\end{abstract}

\section{Introduction}

Pairs of transverse contact distributions on a 3-manifold have been an active topic of interest primarily due to their relations to {\em projective Anosov flows} (see \S 2.2 of \cite{ThurstEl}, and \S 4 of \cite{Mits}), and in particular Anosov flows (eg \cite{Hoz}, \cite{Mass}). Here we examine such pairs admitting infinitesimal symmetries, where:
\begin{dfn}
{\em A {\em contact distribution}, $\xi$, on a 3-manifold $M$ is a non-integrable rank two distribution on $M$. We write $(M, \xi, \tilde \xi)$ for a pair of transverse contact distributions $\xi, \tilde \xi$ on the 3-manifold $M$. A {\em symmetry} of the pair is a diffeomorphism $\varphi:M\to M$ preserving both contact distributions:
    \[ \varphi_*\xi = \xi , ~~~\varphi_*\tilde \xi  = \tilde \xi.\]
    A {\em infinitesimal symmetry} of the pair is a vector field whose flow is by symmetries of the pair.}
\end{dfn}

\begin{dfn}
    {\em We call the {\em local symmetry algebra}, $\mathfrak{g}$, of $(M, \xi, \tilde\xi)$ around a point $m\in M$, the Lie algebra of infinitesimal symmetries of the pair defined in some neighborhood of $m\in M$.}
    \end{dfn}
    
Note that the local symmetry algebra around a point is a real vector space, having a Lie algebra structure by Lie bracket of vector fields.

As is well known, a {\em single} contact distribution on a 3-manifold admits many infinitesimal symmetries: all are locally equivalent to the standard Darboux normal form: $\ker\{dy - pdx\}, ~(x,y,p)\in \R^3$, and any function $H(x,y,p)$ generates an infinitesimal symmetry via its contact Hamiltonian vector field, eg:
\begin{equation}\label{eq:ContHamVf}
    H_p\del_x + (p H_p - H  )\del_y - (H_x + p H_y)\del_p 
\end{equation} 
where here, and in what follows, we write subscripts for partial derivatives ($H_p = \del_p H$, etc.).

On the other hand, as it turns out (see \cite{Wilk}), a typical {\em pair} of transverse contact distributions will admit {\em no} infinitesimal symmetries. As for those pairs of transverse contact distributions admitting infinitesimal symmetries, we find here the following local normal forms:

\begin{thm}\label{thm:lnf} Suppose a pair of transverse contact distributions $\xi, \tilde \xi$ on a 3-manifold $M$ admits an infinitesimal symmetry around $m\in M$ transverse to $\xi\cap\tilde \xi$. Then around $m$, the pair is given in some coordinates by one of the following local normal forms:

\begin{enumerate}
    \item[(I$_1$)] For $c\in \R\backslash \{ 0,1\}$ a constant: \[ \xi = \{ dy = pdx\}, ~~~\tilde \xi = \{ dy = c p dx\}.\]
    The symmetry algebra of such a pair is given by the vector fields of the form:
    \[\mathfrak{g} = \langle u(x) \del_x + v(y)\del_y +  p(v'(y) - u'(x))\del_p\rangle \]
    where $u(x), v(y)$ are any smooth functions of $x$ and $y$ respectively, and $u', v'$ their derivatives.

     \item[(I$_2$)] For $c^2\le 4$ a constant,
      \[ \xi = \{ dy = pdx\}, ~~~\tilde \xi = \{ (p+c)dy + dx = 0\}.\]
      The symmetry algebra of such a pair in these coordinates are those vector fields of the form: 
       \[\mathfrak{g} = \langle u\del_x + v\del_y - u_y(1 + cp + p^2)\del_p \rangle \]
      where $u(x,y)$ is any solution to
      \[ u_{xx} + u_{yy} = c u_{xy} \]
      and $v(x,y)$ its conjugate solution determined, up to addition of a constant, through:
      \[ dv = -u_y dx + (u_x - c u_y)dy. \]
    
     \item[(II$_1$)] For $f(y)$ a function:
     \[ \xi = \{ dy = pdx\} , ~~\tilde \xi = \{ (p +f(y)) dy \pm dx = 0 \}, ~~~\mathfrak{g} = \langle \del_x\rangle. \]

    \item[(II$_2$)] For $f(y)$ a function:
    \[ \xi = \{ dy = pdx\}, ~~~\tilde \xi = \{ dy = f(y) pdx\}, ~~~\mathfrak{g} = \langle u(x)\del_x - pu'(x)\del_p \rangle. \]
In the symmetry algebra, $u(x)$ is any smooth function of $x$ and $u'$ its derivative.

    \item[(III$_1$)] For functions $a(y), b(y), f(p)$: 
    \[ \xi = \{  dy = \frac{p + a(y)}{p + b(y)} dx \} , ~~\tilde \xi = \{ dy = \frac{f(p) + a(y)}{f(p) + b(y)} dx \}, ~~~\mathfrak{g} = \langle \del_x \rangle. \]

    \item[(III$_2$)] For a function $f(p)$:
    \[ \xi = \{ dy = p dx\}, ~~~\tilde \xi = \{ dy =  f(p)dx\}, ~~~\mathfrak{g} =  \langle \del_x, \del_y, x\del_x + y\del_y\rangle. \]

    \item[(IV)] For a function $f(y,p)$:
    \[ \xi = \{ dy = pdx\}, ~~\tilde \xi = \{ dy  = f(y,p)dx\}, ~~~\mathfrak{g} = \langle \del_x\rangle. \]
    
  \end{enumerate}
\end{thm}

\begin{rem} {\em These normal forms are determined up to re-labelings of the contact structures, and the open set containing $m\in M$ is identified in each case with an appropriate open set $U\subset \R^3$ (namely one where $\xi, \tilde\xi$ are transverse contact distributions in these coordinates). For example in case I$_1$, a neighborhood of $m\in M$ is identified with an open subset $U\subset \{ p\ne 0\} \subset\R^3$, or for I$_2$, with $U$ an open subset of the origin $x = y = p = 0$. In the functionally dependent families II - IV, the functions are not arbitrary but subject to certain restrictions in order to define transverse contact structures on some open set. Explicitly, for II$_1$ we may take a neighborhood $U$ of the origin in $\R^3$, for II$_2$: we can take $U$ some neighborhood of $x = y = 0, p = 1$ for $f(y)\ne 0,1$ for $y$ near 0. For types III we may take $U$ as a neighborhood of the origin, where we have $a(y)\ne b(y)$ for $y$ near zero and $f(p)\ne p$ having no fixed points for $p$ near zero and $f'(p)\ne 0$ for $p$ near zero. For type IV, we may take some neighborhood of $x = y = 0, p = 1$ where we have $f(y,p)\ne p$ and $f_p(y,p)\ne 0$ around $y = 0, p = 1$.}
\end{rem}

\begin{rem}[On the transversal condition to the axis.]
{\em For $(M, \xi, \tilde \xi)$ admitting a non-trivial (non-vanishing) infinitesimal symmetry $Y$, the locus where $Y \in \xi\cap \tilde \xi$ is tangent to this `axis', $\xi\cap \tilde\xi$, is of codimension at least one (and in general of codimension two, locally given by an isolated integral curve of this axis) being the intersection of the hypersurfaces $\{ Y\in \xi\} \cap \{ Y\in \tilde \xi\}$ (as $Y\ne 0$ these are locally hypersurfaces as follows from \eqref{eq:ContHamVf}). It is possible, in certain cases, that the hypersurfaces $\{Y\in \xi\}$ and $\{Y\in \tilde\xi\}$ coincide, for example: $\xi = \{ dy = pdx \}, ~\tilde\xi = \{ dy = \pm pdx + dp \}$ and $Y = \del_x$. Here we will be focused on normal forms around `regular' points: where such a symmetry field is transverse to this axis $\xi \cap \tilde \xi$. It would be natural in further work to examine further local normal forms around such `singular' points, for example: where a pair admits a symmetry tangent to $\xi\cap\tilde\xi$ at said point, or around a point at which a symmetry field vanishes.}
\end{rem}

\begin{rem}
{\em We would also like to note that our choice of coordinates in which we present the local normal forms above is somewhat arbitrary: only determined up to arbitrary diffeomorphisms. The forms presented in the above theorem \ref{thm:lnf} are merely those which appeared to us to have the `simplest' defining formulas.}
\end{rem}

The items in theorem \ref{thm:lnf} can be distinguished by considering certain local invariants associated to a pair $(M, \xi, \tilde\xi)$ in order to be certain the items I--IV can be realized by truly {\em distinct} pairs (not equivalent under any local diffeomorphism). These local invariants can be determined using Cartan's equivalence method, as carried out in \cite{Wilk}. See table \ref{table:Summary} below, for a summary of these conditions (our labelling into the types I -- IV being motivated by the different types of invariants involved in this table), as well as remark \ref{rem:DistBasic} below for some of their simpler distinguishing properties.

The paper is structured as follows: in \S \ref{sec:Invs} we will define some basic invariants of a pair $(M,\xi, \tilde \xi)$ of transverse contact structures, and in \S \ref{sec:Comps} carry out the relevant computations to establish the normal forms stated in the main theorem \ref{thm:lnf}, introducing the additional invariants which may be used to distinguish them along the way. In \S \ref{sec:Exs} we will give a pair of examples in which the normal forms of items I$_1$ (with $c = -1$) and I$_2$ (with $c = 2$) are realized: relating to Anosov flows and integrable contact flows respectively.

\section{Invariants}\label{sec:Invs}

We define some basic invariants of a pair of transverse contact structures. First:

\begin{dfn}
    {\em For $(M, \xi, \tilde \xi)$ we call the line field
    \[ A = \xi\cap \tilde \xi \]
    on $M$ the {\em axis} of the pair of transverse contact structures.}
\end{dfn}

\begin{rem}
     {\em We regard $A\to M$ as a line bundle over $M$, with dual line bundle $A^*\to M$, and in particular a line bundle of quadratic forms on $A$ we denote: $(A^*)^2 = A^*\otimes A^*\to M$.}
\end{rem}

Recall that the integral curves of a Legendrian foliation of a given contact structure (a foliation by curves tangent to the given contact structure) inherit a projective structure: see figure \ref{fig:ProjStr} below, or eg \S 4.2.1 of \cite{DSIV}, where:

\begin{dfn}
    {\em A projective structure on an $n$-dimensional manifold $N$ is an atlas, $\fe_a : U_a \to V_a \subset \mathbb{RP}^n$, whose transition functions are by projective transformations.}
\end{dfn}

Since the integral curves of the axis are a Legendrian foliation with respect to two contact structures, they inherit a pair of (in general distinct) projective structures. Let us recall  (see for example \cite{TabOvSch}):

\begin{dfn}
   {\em  For $f:\mathbb{RP}^1\to \mathbb{RP}^1$ a local diffeomeorphism at $x\in \mathbb{RP}^1$, its {\em Schwartzian derivative} $S_x(f)$ at $x\in \mathbb{RP}^1$ is the quadratic form on $T_x(\mathbb{RP}^1)$ given via the cross-ratio expansion: 
    \[ [f(x), f(x_1); f(x_2), f(x_3)] = [x, x_1;x_2, x_3] + \e^2 S_x(f)(v) + O(\e^3),\] where $v \in T_x(\mathbb{RP}^1)$ is extended to a vector field around $x$ with flow $\fe_t$ and $x_j = \fe_{j\e}(x)$. Its formula in an affine chart: $(1:p) \mapsto (1: f(p))$ is:
    \[ S(f) = \left( \left( \frac{f''}{f'}\right)' - \frac12 \left( \frac{f''}{f'}\right)^2 \right) (dp)^2. \]}
\end{dfn}

Note that $S(f)\equiv 0$ exactly when $f$ is a projective transformation (in an affine chart: $f(p) = \frac{ap + b}{cp + d}$).

\begin{figure}[h]
\begin{tikzpicture}

\draw (-1,0) -- (3,0);
\draw (-2,-1) -- (2,-1);
\draw (-2,-1) -- (-1,0);
\draw (2,-1) -- (3,0);

\draw (2.5,.5) -- (2.5,2);
\draw (.5,.5) -- (.5,2);

\draw (1.5,.1) -- (1.5,1.6);
\draw (-.5,.1) -- (-.5,1.6);

\node at (1.5,2) {$\mathcal{A}$};

\draw[blue, very thick] (1.2,1) -- (1.2,.3);
\draw[blue, very thick] (1.8,1.3) -- (1.8,.6);

\draw[blue, very thick] (1.2,1) -- (1.8,1.3);
\draw[blue, very thick] (1.2,.3) -- (1.5,.45);

\draw[blue, very thick] (1.7,.55) -- (1.8,.6);

\node[blue] at (2.15,1) {$\xi_m$};

\draw[red] (1.3,1.3) -- (1.3,1.1);

\draw[red] (1.7,1.1) -- (1.7,.35);

\draw[red] (1.3,1.3) -- (1.7,1.1);

\draw[red] (1.5,.45) -- (1.7,.35);

\node[red] at (.95,1.4) {$\tilde \xi_m$};

\filldraw [black] (1.5,.8) circle (2pt);

\draw[blue, very thick] (1.2,-.6) -- (1.8,-.3);

\draw[red] (1.2, -.3) -- (1.8, -.6);  

\filldraw [black] (1.5,-.45) circle (2pt);
\node at (.8,-.4) {$\pi(m)$};


\node at (4.5,1.5) {$M$};

\draw[->](4.5,1) -- (4.5,0);
\node at (4.7,.5) {$\pi$};

\node at (4.5,-.5) {$M/A$};


\end{tikzpicture}

\caption{Locally, integral curves $\mathcal{A}$ of a Legendrian foliation on $(M,\xi)$ are canonically  identified with $\mathbb{RP}^1$ (up to projective transformations). For $\Sigma \cong_{loc} M/A$, a local slice, send $m\in \mathcal{A}$ to $\pi_*(\xi_m) \in \mathbb{P}(T_{\pi(m)}\Sigma)$ (by the contact condition, a local diffeomorphism).}
\label{fig:ProjStr}
\end{figure}
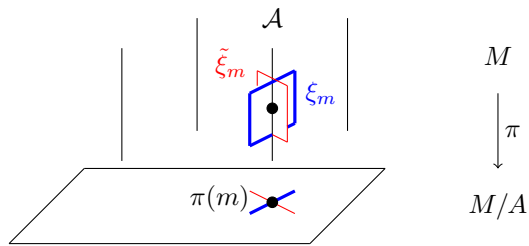

\begin{prop}\label{prop:Schwartz}
    A pair of transverse contact structures, $\xi, \tilde \xi$, with axis $A$ determines a section:
    \[ \mathcal{S}: M \to (A^*)^2 \]
    given, at each point $m\in M$, as the Schwartzian derivative at $m$ between the induced projective structures on the integral curve of the axis through $m$. 
    We call this section the {\em Schwartzian invariant} of $(M,\xi, \tilde \xi)$.
\end{prop}
\begin{proof}
    Explicitly, this invariant is defined as follows: let $\mathcal{A}_m$ be the integral curve of $A$ through $m$, and consider the maps $\delta_m, \tilde\delta_m : \mathcal{A}_m \to \mathbb{P}(T_mM/A_m)$ sending $a\in \mathcal{A}_m$ to $\fe_{-t,*}\xi_a$ (resp.~$\fe_{-t,*}\tilde \xi_a$) where $a = \fe_t(m)$ for $\fe_t$ the flow of some vector field $X$ spanning $A$. By the contact condition on $\xi, \tilde\xi$ these maps are local diffeomorphisms. Then pull back by $\delta_m$ the Schwartzian derivative of $\tilde\delta_m\circ\delta_m\inv:\mathbb{P}(T_mM/A_m) \to \mathbb{P}(T_mM/A_m)$ at $\delta_m(m)$ to have a quadratic  form $\mathcal{S}_m$ on $A_m = T_m\mathcal{A}_m$.
\end{proof}

\begin{coro}\label{rem:Schwartz}
    The Schwartzian invariant vanishes, $\mathcal{S}\equiv 0$, iff the integral curves of the axes inherit a {\em common} projective structure from the pair of contact distributions.
    In local coordinates $(x,y,p)$ for which $\xi = \{ dy = pdx\}, ~~\tilde \xi = \{ dy = f(x,y,p)dx\}$,
    we have:
    \[ \mathcal{S} = \left( \del_p\left(\frac{f_{pp}}{f_p}\right) - \frac12 \left(\frac{f_{pp}}{f_p}\right)^2  \right) (dp)^2.\]
     The vanishing of $\mathcal{S}$ being equivalent to $f(x,y,p)$ having the form: $f = \frac{a(x,y)p + b(x,y)}{c(x,y)p + d(x,y)}$.
\end{coro}

To define further invariants, we will make some orientation conventions. Recall that a contact structure $\xi$ on the 3-manifold $M$ induces an orientation on $M$ given locally by $\alpha\wedge d\alpha$ for some (local) choice of contact 1-form with $\xi = \ker\alpha$. Since $\dim M = 3$ (or more generally for a contact structure on $M$ with $\dim M = 4k-1$) this (local) orientation is independent of arbitrary scalings of $\alpha$, and so defines a global orientation on $M$ induced by $\xi$.
\begin{dfn}
   {\em  We call a pair $(M, \xi, \tilde\xi)$ {\em commonly (positively) oriented} (resp.~{\em oppositely (negatively) oriented}) when the orientations induced by the contact distributions coincide (resp.~are opposite). We denote such pairs by:
\[ (M, \xi, \tilde \xi)_+, ~~~(M, \xi, \tilde \xi)_-\]
for common orientations (`$+$') and opposite orientations (`$-$').}
\end{dfn}
Note that an isolated contact structure $(M,\xi)$ can only have a sign relative to some choice of orientation on $M$.
The case of an oppositely oriented pair, $(M, \xi, \tilde\xi)_-$, usually called a {\em bi-contact structure}, is the relevant case for projective Anosov flows, where the integral curves of the axis are the integral curves of the projective Anosov flow. Following \cite{Mass} \S 2:

\begin{dfn}\label{def:BalGenInv}
{\em We call a pair $(\alpha, \tilde\alpha)$ of contact forms $\ker \alpha = \xi, \ker\tilde\alpha = \tilde \xi$ for $(M, \xi, \tilde \xi)_\pm$ {\em balanced} when:
    \[ \alpha\wedge d\alpha = \pm \tilde\alpha\wedge d\tilde\alpha. \]
    Supposing the axis has been oriented by $A = \text{span}_+\{ X\}$, we take the {\em normalized vector field} $X:M\to A$ spanning the axis through:
    \[ i_X(\alpha\wedge d\alpha) = \alpha\wedge\tilde\alpha \]
    for $\alpha, \tilde\alpha$ a choice of balanced contact forms. The corresponding one form along the axis we denote as, $ds : M\to A^*$,
    with $ds(X) = 1$. For a function $f:M\to \R$, we denote its {\em derivative along the axis} by:
    \[ f' := Xf :M\to \R.\]
    As well, we have an induced function, which we call the  {\em generating invariant}:
    \[ I:M\to \R, ~~~I := \frac{\alpha\wedge d\tilde\alpha + \tilde\alpha \wedge d\alpha}{\alpha\wedge d\alpha} \]
    where $\alpha, \tilde\alpha$ are balanced contact forms such that $i_X(\alpha\wedge d\alpha) = \alpha\wedge\tilde\alpha$.}
\end{dfn}

\begin{rem}
    {\em The conditions of being balanced as well as $i_X(\alpha\wedge d\alpha) = \alpha\wedge \tilde\alpha$ determine the contact forms $\alpha, \tilde\alpha$ up to simultaneous scalings: $\alpha, \tilde\alpha \to \lambda\alpha, \lambda\tilde\alpha$ for some function $\lambda \ne 0$. The 1-form $ds$ along the axis, and function $I$, are independent of such simultaneous scalings (only their sign depending upon choice of axis orientation).}
\end{rem}

\begin{rem}
    {\em The combination $\alpha\wedge d\tilde\alpha + \tilde\alpha\wedge d\alpha$ used here to define the generating invariant $I$, plays a main role already in \cite{Geig}, and \cite{ThurstEl} (see eq.~(2.2) of ch.~2).}
\end{rem}

\begin{rem}\label{rem:altForms}
{\em For any 1-form $\beta$ with $\beta|_A = ds$, and balanced contact forms $\alpha, \tilde\alpha$, we have:
\[ \alpha\wedge d\alpha = \alpha \wedge\tilde\alpha\wedge\beta, \]
since a dual frame to $\alpha, \tilde\alpha, \beta$ with $\beta|_A = ds$ is of the form $X_1, X_2, X$ for $X$ the normalized vector field along the axis. Consequently:
\[ f' = \frac{\alpha\wedge\tilde\alpha \wedge df}{\alpha\wedge d\alpha}. \] }
\end{rem}

In local coordinates, we have the following explicit formulas:

\begin{prop}[explicit formulas]\label{rem:explFormulasGen}
    For $(M, \xi, \tilde \xi)$ given in local coordinates by $\xi = \{ dy = pdx\}, \tilde \xi = \{ dy = f(x,y,p)dx\}$, (and axis oriented by $\text{span}_+\{ \del_p\}$) we have:
    \[ ds = \frac{|f_p|^{1/2}}{|p-f|} dp, ~~X = \frac{|p-f|}{|f_p|^{1/2}} \del_p, \]
    \[ I = |f_p|^{-1/2} \left( 1 + f_p + \frac12 (p-f) \frac{f_{pp}}{f_p} \right) \frac{p-f}{|p-f|}. \]
    In the special case $\mathcal{S}\equiv 0$, with $f = \frac{a(x,y) p + b(x,y)}{c(x,y)p + d(x,y)}$, then: $I = \pm \frac{a+d}{\sqrt{|ad-bc|}}$.
\end{prop}

\begin{proof}
The contact forms
\begin{equation}\label{eq:BalCoords}
    \alpha = dy - pdx, ~~\tilde\alpha = |f_p|^{-1/2}\left( \frac{p-f}{|p-f|}\right) (dy - f dx)
\end{equation} 
are balanced contact 1-forms for $\xi =\{ dy = pdx\}, \tilde\xi = \{ dy = fdx\}$. The stated formulas follow then by differentiating \eqref{eq:BalCoords} and using definition \ref{def:BalGenInv}.
\end{proof}

\begin{rem}
     {\em One can also verify (for example in local coordinates of the last proposition \ref{rem:explFormulasGen}) that the Schwartzian invariant of proposition \ref{prop:Schwartz} is given by: 
\begin{equation}\label{eq:SchwartzI'}
    \mathcal{S} = 2I' (ds)^2.
\end{equation} 
In particular, for any integral curve, $\gamma(t)$, with $\dot\gamma = X$ of the normalized vector field $X$ along the axis, we have:
\[ \frac12 \int_{a}^{b} \mathcal{S}(\dot\gamma)~dt = I(\gamma(b)) - I(\gamma(a)). \]
For example in case one has $\mathcal{S} > 0$ everywhere, the integral curves of the axis admit {\em no} periodic orbits (if $M$ is compact however, there are always zeroes of $\mathcal{S}$, eg at any critical point of $I$).}
\end{rem}

\begin{rem}\label{rem:DistBasic}
    {\em Already these basic invariants can be used to distinguish between several of the types presented in theorem \ref{thm:lnf}. For example, I$_1$ with $c = -k ^2 <0$ are all oppositely oriented and have $I = k - \frac{1}{k} = cst.$, whereas I$_1$ with $c = k^2 > 0$ are commonly oriented pairs with $I = k + \frac{1}{k} = cst. > 2$. In type I$_2$ we have {\em only commonly oriented pairs} with $I^2 = c^2 \le 4$. Types I, II all have $\mathcal{S}\equiv 0$, so that the types III, IV with $\mathcal{S}\ne 0$ are then distinct from I, II. Distinguishing between the remaining subtypes will be a main focus of \S \ref{sec:Comps}, by using some finer invariants of such pairs (see table \ref{table:Summary}).
    
    We also note that types II-IV may be realized by both commonly (`$+$') or oppositely (`$-$') oriented pairs. Namely type II$_1$ according to the sign $\pm$ in $\tilde\xi = \{ (p + f(y))dy \pm dx = 0\}$, type II$_2$ according to the sign of $f(y)$, types III according to the sign of $f'(p)$, and type IV according to the sign of $f_p$.}
\end{rem}

\section{Two examples}\label{sec:Exs}

Before beginning the computations of \S \ref{sec:Comps} to establish the normal forms of theorem \ref{thm:lnf}, we mention a pair of examples.

\subsection{Anosov flows}\label{sec:An}
The most well-known instances of pairs of transverse contact structures are the {\em bi-contact structures}: when the pairs induce opposite orientations ($(M, \xi, \tilde \xi)_-$ in the notation here), being of grand interest due to their relations to projective Anosov flows. In particular, for {\em orientable} Anosov flows (those whose invariant stable and unstable foliations are orientable) it has been established (theorem~1.1 of \cite{Hoz} and lemma 2.4, and \S 3.2 of \cite{Mass}) :

\begin{teon}
    {\em Let $X$ be a smooth orientable Anosov vector field on a compact 3-manifold $M$. Then there exist a pair of transverse and oppositely oriented $C^1$ contact distributions:
    \[ \xi = \ker \alpha, ~~\tilde \xi = \ker \tilde\alpha \]
    with $X\in \xi\cap \tilde \xi$, and balanced $C^1$ contact forms $\alpha, \tilde\alpha$
    so that $\ker\{\alpha + \tilde\alpha\}, ~\ker\{\alpha - \tilde\alpha\}$ are integrable distributions on $M$ (namely the foliations by weak stable/unstable manifolds of the Anosov flow).}
\end{teon}

From which, in relation to theorem \ref{thm:lnf}, we find:

\begin{corn}
    {\em Let $M$ be a compact 3-manifold, and $X$ a smooth orientable Anosov vector field on $M$. Then there exists a pair of oppositely oriented transverse contact structures $(M, \xi, \tilde \xi)_-$ on $M$ with $X\in \xi\cap \tilde \xi$ such that, around any point $m\in M$, the pair has the local form of item I$_1$ with $c = -1$ of theorem \ref{thm:lnf}:
    \[ \xi = \{ dy = p dx\}, ~~\tilde \xi = \{ dy = - p dx\}. \]
    where the coordinates $x,y$ are $C^1$ and $p$ a $C^1$ coordinate along integral curves of $X$.}
\end{corn}
\begin{proof}
    Since the balanced $C^1$ contact forms of the previous theorem have $\alpha + \tilde\alpha, \alpha - \tilde\alpha$ defining integrable distributions, we have $I\equiv 0$, which corresponds to type I$_1$ with $c = -1$. Because of the low regularity in this case, we should take some care to keep track of the regularity of the coordinates, so we recall the construction from proposition \ref{prop:Icst}. We first choose locally some $C^1$ functions $x,y$ so that $\alpha + \tilde\alpha = a dx, ~~\alpha - \tilde\alpha = b dy$, for some (in general $C^0$) non-vanishing functions $a, b$. Then we have in this neighborhood: $(\alpha, \tilde\alpha)$ proportional to $(dy - pdx, ~ dy + pdx)$ for $p = -a/b$. Since $X$ and its flow $\fe_t$ are smooth, and $\xi, \tilde\xi$ are $C^1$, the map: $\fe_t(m) \mapsto  \fe_{-t,*}\xi_{\fe_t(m)} \in \mathbb{P}(T_mM/X_m)$ is $C^1$ (and locally invertible). Consequently, for each $m$ in this neighborhood, $\R\ni t\mapsto p(\fe_t(m))\in \R$ is also $C^1$ and locally invertible around $t = 0$.
\end{proof}

In other words, Anosov flows can be (globally) supported by the `simplest' types of bi-contact structures (a bi-contact structure is said to be {\em supporting} a vector field $X$ if $X\in A = \xi\cap \tilde \xi$). However, we note that in general none of these local symmetries need extend to give any global symmetries of such a pair.

\begin{rem}\label{rem:RegAnosov}  {\em Because of the low ($C^1$) regularity of the coordinates in the last corollary, rather than a symmetry algebra it is more appropriate here to state the local symmetry group, namely the transformations: 
    \[ (x,y,p)\mapsto (F(x), G(y), \frac{G'(y)}{F'(x)}p).\]
Even in cases with higher regularity (if this regularity is high enough, such flows are known to be topologically equivalent to `algebraic Anosov flows': quotients of geodesic flow on the hyperbolic plane, or suspension of a hyperbolic linear diffeomorphism of a torus, see \cite{Ghys}, \cite{Ghys2}), we find the last corollary rather surprising, since the local symmetry groups are {\em infinite dimensional} (depending upon two arbitrary invertible functions of a single variable) rather than the naive expectation of say a mere 3-dimensional symmetry algebra ($\mathfrak{sl}_2(\R)$).}
\end{rem}

There are many bi-contact structures supporting a given projective Anosov flow. It would be interesting to examine to what extent more general projective Anosov flows might be globally supported by other such `locally symmetric' items of theorem \ref{thm:lnf}.

For instance, one might examine to what extent those projective Anosov flows which are {\em not} Anosov might {\em fail} to be globally supported by item I$_1$. For example, an oppositely oriented pair $(M, \xi, \tilde\xi)_-$ with $I = cst.$ (ie type I$_1$ with $c<0$) has associated to it a pair of invariant foliations containing the axis (see proposition \ref{prop:IntDist} below). It is known (see example 2.2.9 of \cite{ThurstEl}) that there exist projective Anosov flows which are {\em not} contained in any invariant (rank 2) foliations. Thus, such projective Anosov flows  {\em cannot} be given globally as an intersection of a pair of oppositely oriented contact structures with generating in variant, $I$, constant.

\subsection{Integrable contact vector fields}
On the other hand, we recall a definition of contact integrability\footnote{A definition can also be given for general contact vector fields: \cite{KhTabContInt}. Here we consider, for simplicity, the Reeb case.} :

\begin{dfn}[\cite{KhTabContInt}]
    {\em For a 3-dimensional contact manifold $(M,\xi)$, a vector field $Y$ {\em transverse} to $\xi$ and whose flow preserves $\xi$ (a {\em Reeb field} of $\xi$) is said to be {\em contact integrable} if there exists a Legendrian foliation $\mathcal{G}$ ($T\mathcal{G}\subset \xi$) of $(M, \xi)$ preserved by $Y$.}
\end{dfn}
\begin{rem}
    {\em The Legendrian foliation $\mathcal{G}$ being preserved by $Y$ is equivalent to the distribution $T\mathcal{G}\oplus \langle Y\rangle = T\mathcal{F}$ being integrable (such a foliation $\mathcal{F}$ is called a {\em co-Legendrian} foliation of $(M, \xi)$).}
\end{rem}

In particular, since any symmetry field $Y$ of $(M, \xi, \tilde \xi)$ will preserve the Legendrian foliation by the axes $A = \xi\cap \tilde \xi$, any such symmetry field, transverse to $\xi$ say, is contact integrable in the above sense. One can expect such a symmetry field to then have associated integrals by any invariants (eg $I$ of definition~\ref{def:BalGenInv}) of the pair $(M, \xi, \tilde \xi)$. Conversely:
\begin{prop}
    Suppose $Y$ is a contact integrable Reeb field of $(M, \xi)$, and the associated co-Legendrian foliation is given by:
    \[ T\mathcal{F} = T\mathcal{G}\oplus \langle Y\rangle = \ker \gamma \]
    for $\gamma$ some {\em closed} 1-form on $M$. Then there exists a commonly oriented contact structure $\tilde \xi$ on $M$ transverse to $\xi$, also preserved by $Y$, and so that $(M, \xi, \tilde \xi)_+$ is at any point of $m\in M$ locally given by item I$_2$ of theorem \ref{thm:lnf} with $c = 2$:
    \[ \xi = \{ dy = pdx\}, ~~\tilde \xi = \{ dy + \frac{dx}{p+2} = 0\}. \]
\end{prop}
\begin{proof}
    Let $\alpha$ be a contact form for $\xi$ having $Y$ as its Reeb field ($i_Y\alpha \equiv 1, i_Yd\alpha \equiv 0$), and note that $\gamma\wedge d\alpha \equiv 0$. Taking $\tilde\alpha = \alpha + \gamma$ we find: $\tilde\alpha\wedge d\tilde\alpha = \alpha\wedge d\alpha$, so we have balanced contact forms $\alpha, \tilde\alpha$, for the pair $(M, \xi, \tilde \xi)_+$ where $\tilde \xi = \ker\tilde\alpha$. This pair has $I \equiv 2$, so by proposition~\ref{prop:Icst}, is locally equivalent around any point to item I$_2$ of theorem \ref{thm:lnf} with $c = 2$. Moreover: $L_Y\tilde\alpha = d(i_Y\alpha) = 0$, so that $Y$ is an infinitesimal symmetry of this pair $\xi, \tilde \xi$.
\end{proof}

\begin{rem}
   {\em The condition that $\gamma$ be closed in the last proposition is rather restrictive. For example, the holonomy of $\mathcal{F}$ must be trivial. It would be natural then to explore weakening this condition in future work.}
\end{rem}

A contact integrable vector field $Y$ might as well preserve some other pairs $(M, \xi', \tilde \xi')$ different from those of the previous proposition. It would be interesting to examine which items of theorem \ref{thm:lnf} may arise in some more specific examples of integrable contact vector fields (in particular those items with non-trivial invariants corresponding to non-trivial integrals of the system).

\section{Computations}\label{sec:Comps}

Here we will determine the normal forms of theorem \ref{thm:lnf} (see \S \ref{sec:prf} for a summary). In this section, we will not concern ourselves with the regularity of the contact structures (everything will be assumed as differentiable as necessary). For many of the proofs, one may be able to get by with less regularity (as we have seen above for the Anosov flows example \S \ref{sec:An}).  First, we have:

\begin{prop}\label{prop:genLNF}
    Suppose $(M, \xi, \tilde \xi)$ admits an infinitesimal symmetry transverse to the axis at $m\in M$. Then there are local coordinates in which: $\xi = \{ dy = pdx\} , ~\tilde \xi = \{ dy = f(y,p)dx\}$.
\end{prop}
\begin{proof}
If a pair $(M, \xi, \tilde\xi)$ admits a symmetry field $Y$ transverse to its axis $\xi\cap\tilde\xi$, we can consider the various cases according to whether at $m\in M$: (i) $Y_m$ is transverse to both $\xi, \tilde\xi$ (the `generic' type of point), or (ii) $Y_m\in \tilde\xi_m$, or (iii) $Y_m\in\xi_m$. Clearly, up to re-labelling of the pairs it is enough to consider the cases (i) and (ii): both having, at $m\in M$, that $Y_m$ is transverse to $\xi_m$.

    We may always take local Darboux coordinates $(x,y,p)$ in which $\xi = \{ dy = pdx\}$ and the axis is rectified as $\langle \del_p\rangle$ so that its pair is given by $\tilde\xi = \{ dy = f(x,y,p)dx\}$. Note that any contact transformation of $\xi$ preserving this axis $\langle\del_p \rangle$ projects to a transformation of the $xy$-plane, and conversely any transformation of the $xy$-plane lifts (its `{\em contact lift}') to a contact transformation of $\xi = \{ dy = pdx\}$ (view $(x,y,p)$ as coordinatizing the contact elements of the $xy$-plane with standard contact structure $dy = pdx$, explicitly $(x,y)\mapsto (F(x,y), G(x,y))$ lifts to $(x,y,p)\mapsto (F(x,y), G(x,y), \frac{G_x + G_y p}{F_x + F_y p})$). Such contact transformations are generated by the contact lifts of vector fields $u\del_x + v\del_y$ on the $xy$-plane:
    \[ u(x,y)\del_x + v(x,y) \del_y + w \del_p , ~~~w = v_x + p v_y - p( u_x + p u_y) \]
    for such a vector field to also preserve $dy = f(x,y,p)dx$ we have the further condition:
    \begin{equation}\label{eq:symm}
        0 = f(v_y - fu_y) - (v f_y + w f_p) + v_x - (fu)_x.
    \end{equation}
    Suppose then that such a pair $\xi, \tilde \xi$,  admits an infinitesimal symmetry transverse to the axis. Then, since this infinitesimal symmetry is transverse to the axis, $\langle \del_p\rangle$, it projects to a {\em non-singular} vector field on the $xy$-plane. We then consider a change of coordinates on the $xy$-plane $(x,y)\mapsto (s, t)$ rectifying this symmetry field to say $\del_s$ in these new coordinates. Without loss of generality, we may always assume at a given point that we have chosen our transversal curve to the symmetry field (the $t$-axis) to be transverse to the pair of contact structures at this point (analytically: we may apply transformations $(s,t)\mapsto (s+f(t), g(t)) = (S, T)$ with, still, the symmetry field $\del_s\mapsto \del_S$). Under a contact lift of such a change of coordinates, $(x,y,p)\mapsto (s,t,\s)$ with the $t$-axis transverse to the contact structures at the correspondent point to $m\in M$, we then have the form $dt = \s ds, dt = F(s,t,\s)ds$ where for cases (i), (ii) we have that $m$ corresponds some point with $\s\ne 0$ (for case (iii) some point with $\s = 0$). In any case, we have a symmetry field $\del_s$ of the pair $dt = \s ds, dt = F ds$ and equation \eqref{eq:symm} reads $F_s \equiv 0$, ie $F(t,\s)$ is independent of $s$ as claimed.
\end{proof}

This simple computation establishes our general normal form of a pair admitting a symmetry (item IV of theorem \ref{thm:lnf}) and includes all the other items as certain sub-cases. However it is not clear from this last proposition \ref{prop:genLNF} whether such a pair might admit {\em more} infinitesimal symmetries. The main thrust of theorem \ref{thm:lnf}, and the following computations in this section, is to give a finer description of distinct subcases of proposition \ref{prop:genLNF}, and in particular their explicit characterizations in terms of relevant invariants of the pair. For this finer description, we will often use (following directly from definition~\ref{def:BalGenInv}):

\begin{prop}
    Let $\alpha, \tilde\alpha$ be balanced contact forms for $(M, \xi, \tilde \xi)_\pm$ and $\beta$ a 1-form on $M$ with $\beta|_A = ds$. Then such contact forms satisfy:
    \begin{equation}\label{eq:GenBalStrEq}
        d\alpha = \tilde h \alpha\wedge\tilde\alpha + \tilde\alpha\wedge \beta + g \beta\wedge\alpha , ~~d\tilde\alpha = h \alpha\wedge\tilde\alpha + (I - g )\tilde\alpha\wedge \beta \pm  \beta\wedge\alpha
    \end{equation}
    for some functions $h, \tilde h,  g$ on $M$ (and where $I:M\to\R$ is the generating invariant from definition \ref{def:BalGenInv}). Note that $\alpha, \tilde\alpha$ are determined up to simultaneous scalings: $\alpha, \tilde\alpha\mapsto \lambda\alpha, \lambda\tilde\alpha$, and $\beta$ up to shifts $\beta\mapsto \beta + b_1\alpha + b_2\tilde\alpha$.
\end{prop}

\begin{prop}\label{prop:IntDist}
    Let $\alpha, \tilde\alpha$ be balanced contact forms for $(M, \xi, \tilde \xi)_\pm$. Then the kernel of $\alpha + \rho \tilde\alpha$ determines an integrable distribution containing the axis when the function $\rho$ satisfies:
    \begin{equation}\label{eq:Ric}
        \rho' = 1 + I\rho \pm \rho^2.
    \end{equation}
\end{prop}

\begin{prop}
    If $Y$ is an infinitesimal symmetry of $(M, \xi, \tilde\xi)$ transverse to the axis, then $A\oplus Y$ is an integrable distribution (recall $A = \xi \cap \tilde\xi$ is the axis of the pair).
\end{prop}

The normal forms in theorem \ref{thm:lnf} are then obtained essentially by applying Cartan's equivalence method to find certain co-frames $\alpha, \tilde\alpha, \beta$ as in \eqref{eq:GenBalStrEq} associated to the pair $(M, \xi, \tilde \xi)$, whose coefficents $h,\tilde h, g$ allow us to constrain possible symmetries, and determine appropriate solutions $\rho$ to \eqref{eq:Ric} which allow us to integrate the structure equations \eqref{eq:GenBalStrEq} to find directly the local normal forms stated above.

\begin{rem}\label{rem:counting}
    {\em In `counting' the symmetries we recall (see, eg, Lecture 6 of \cite{Gard}) that if there is a given co-frame (or $e$-structure) $\omega^j$ on some manifold $M$ (so $\omega^1, ..., \omega^n$ are a basis for $T^*M$) then a symmetry $\fe:M\to M$ preserving this co-frame, $\fe^*\omega^j=\omega^j$, is determined by where it sends a single point (by appropriate application of the Frobenius theorem to the graph in $M\times M$ of such a symmetry). Possible symmetries are thus parametrized by where they send a single point. In particular, the symmetry group of the co-frame has dimension at most $\dim M$. The coefficients $I_{kl}^j:M\to \R$ of the structure equations: $d\omega^j = I_{kl}^j\omega^k\wedge\omega^l$ of this co-frame, and their successive derivatives $dI_{kl}^j = I_{kl;m}^j\omega^m$, ...etc... may obstruct such symmetries, by constraining the image of a point to their common level sets. In case one has that the rank, $r$, of the differentials spanned by these invariants and their successive derivatives is constant around some point $m\in M$, then the (local) symmetries of the co-frame around $m$ have dimension: $\dim M - r$.}
\end{rem}

\begin{rem}
    {\em We also recall that a Ricatti equation: $f' = 1 + 2b f \pm f^2$, has to any pair of solutions, $f_1, f_2$, its general solution, $f$, given through: $\frac{f-f_1}{f-f_2} = ce^\lambda$ for $c' = 0$ constant and $\lambda' = \pm (f_1 - f_2)$. In the case of constant coefficients $b'= 0$, each solution $f$, has an associated solution, $f_*$, through: $f_* \pm b  = \frac{b^2 \mp 1}{f \pm b}$ (and the general solution can be given explicitly if need be).}
\end{rem}

To consider some more particular forms of proposition \ref{prop:genLNF}, we will seperately consider the cases when $\mathcal{S}\equiv 0$ and $\mathcal{S}\ne 0$ (see prop.~\ref{prop:Schwartz}).

\subsection{Vanishing Schwartzian}  We first determine local normal forms when the induced projective structures on the integral curves of the axis coincide ($\mathcal{S} \equiv 0$).  As we have already noted above (corollary \ref{rem:Schwartz}), such pairs already have a rather specific form, in fact:

\begin{prop}\label{prop:genS0}
    Suppose $(M, \xi, \tilde \xi)_\pm$ has $\mathcal{S} \equiv 0$. Then there are local coordinates $(x,y,p)$ with:  $\xi = \{ dy = pdx\}, ~\tilde \xi = \{ dy \pm \frac{b^2}{p + bI} dx = 0\}$,
    for some functions $b(x,y), I(x,y)$.
\end{prop}

\begin{proof}
For common orientations, $(M, \xi, \tilde \xi)_+$, choose some (local) non constant ($\rho'\ne 0$) solution $\rho$ to \eqref{eq:Ric}. Since $\alpha + \rho\tilde\alpha$ is integrable, we may take locally some function so that $dx \sim \alpha + \rho\tilde\alpha$, or by re-scaling the balanced contact forms by an appropriate integrating factor have:
    \[ dx = \alpha + \rho \tilde\alpha. \]
   Now, from $I' \equiv 0$ (eq.~\eqref{eq:SchwartzI'}), we have another solution $\rho_* = - C + \frac{C^2 - 1}{\rho + C}$ to \eqref{eq:Ric}, for $C = I/2$, and so $\alpha + \rho_* \tilde\alpha$ is integrable, and we may take:
   \[ dy = \mu (-(\rho + C)\alpha + (C\rho + 1) \tilde\alpha) \]
   for some integrating factor $\mu$ (in fact differentiating these equations and using \eqref{eq:GenBalStrEq}, one finds $d\mu\wedge dx\wedge dy = 0$, so that $\mu(x,y)$). Solving for the contact forms we find:
   \[ \mu\rho'\alpha = \mu(C\rho + 1) dx - \rho dy, ~~\mu\rho' \tilde\alpha = \mu(\rho + C) dx + dy\]
   which can be re-arranged in the stated normal form (since $\mu(x,y), C(x,y)$ there is some integrating factor $M(x,y)$ such that locally $M( dy + \mu C dx) = dY$ is exact, then we may take $p = \mu M (I + \frac{1}{\rho})$, and $b = -\mu M$). The oppositely oriented case can be handled analogously using $\rho_* = C + \frac{C^2 + 1}{\rho - C} = \frac{C\rho + 1}{\rho - C}$.
\end{proof}

To study in more detail the pairs with $\mathcal{S}\equiv 0$, we will introduce some more special co-frames than \eqref{eq:GenBalStrEq}: 

\begin{dfn}\label{def:SchwCofr}
    {\em For $(M, \xi, \tilde \xi)_\pm$ with $\mathcal{S}\equiv 0$, we have to each (local) non-constant solution $\rho$ of \eqref{eq:Ric}, the {\em induced} (local) co-frames:
    \[ \alpha, \tilde\alpha, \beta = d\rho/\rho' ~~~~~ dx, \eta, d\rho \]
    where $\alpha, \tilde\alpha$ are balanced contact forms scaled so that $\alpha + \rho \tilde\alpha = dx$ is (locally) exact. We take 
    \[\eta := (\rho \pm C)\alpha \mp (C\rho + 1)\tilde\alpha\] for $C= I/2$ as an integrable one-form associated to the additional solution $\rho_* = \mp \frac{C\rho + 1}{\rho \pm C}$ of \eqref{eq:Ric}. Given $\rho$, such balanced 1-forms ($\alpha, \tilde\alpha$), as well as the pair ($dx, \eta$), are determined up to simultaneous scalings by functions $\lambda(x)$ depending only on $x$.}
\end{dfn}

For such induced co-frames we compute:

\begin{prop}[structure equations]
    For $(M, \xi, \tilde \xi)_\pm$ with $\mathcal{S}\equiv 0$,  let $\rho$ be some local non-constant solution to \eqref{eq:Ric} with corresponding induced co-frames  from definition \ref{def:SchwCofr}. Then these co-frames satisfy:
    \begin{equation}\label{eq:StrEqsRho}
         d\alpha = -\rho h \alpha\wedge \tilde\alpha + \tilde\alpha \wedge\beta  \mp \rho \beta\wedge\alpha , ~d\tilde\alpha =  h \alpha\wedge \tilde\alpha + (I \pm \rho ) \tilde\alpha \wedge\beta \pm \beta\wedge\alpha , ~d\beta = - \frac{\rho}{\rho'} dI\wedge \beta.
    \end{equation}
    \begin{equation}\label{eq:StrEqsInt}
        d\eta = \left( h dx + \frac{\rho}{\rho'} dI\right) \wedge\eta - \frac{\rho^2\mp 1 }{2\rho'} dI\wedge dx
    \end{equation}
    
\end{prop}

\begin{proof}
    The structure equations \eqref{eq:StrEqsRho} are a direct computation from \eqref{eq:GenBalStrEq} using that $\alpha + \rho\tilde\alpha$ is closed, while \eqref{eq:StrEqsInt} follows by direct differentiation of $\eta$ from definition \ref{def:SchwCofr} (and using \eqref{eq:StrEqsRho}).
\end{proof}

In the special case that the generating invariant $I$ (definition \ref{def:BalGenInv}) is constant:

\begin{prop}\label{prop:Icst}
    Suppose $(M, \xi, \tilde \xi)_\pm$ has generating invariant $I$ constant. Then in some local coordinates we have:
    \[ \xi = \{ dy = pdx\}, ~~\tilde \xi = \{ dy \pm \frac{dx}{p + I} = 0 \}. \]
\end{prop}

\begin{proof} First, observe that if we have $\xi = \{ dy = pdx\}, \tilde\xi = \{ dy = \frac{ap + b}{cp +d}dx\}$ with $a,b,c,d$ constants, then one can take a linear change of the $xy$ coordinates in order to have the stated form $\{ dy = pdx\}, \{ dy \pm \frac{dx}{p+I} = 0\}$ for $I\equiv cst.$. So, to establish the stated local normal form, it suffices to show that when $I\equiv cst.$ we can write locally $\xi = \{ dy = pdx\}, \tilde\xi = \{ dy = \frac{ap + b}{cp +d}dx\}$, with $a,b,c,d$ constants.

For oppositely oriented pairs $(M, \xi, \tilde \xi)_-$, the associated Ricatti equation \eqref{eq:Ric} is {\em hyperbolic}: having two real roots $k, -1/k$ of $k^2 - Ik -1 = 0$ (so $I = k - \frac{1}{k}$). Taking some balanced contact forms so that locally we have $dx = \alpha + k \tilde\alpha, dy = \mu (\alpha - \frac{1}{k}\tilde\alpha)$ for $\mu$ some integrating factor, we find the desired form $\{ dy = pdx\}, \{ dy = - k^2pdx\}$ for $p = -\mu/k^2$. Similarly one can treat the case $(M,\xi,\tilde\xi)_+$ with $I^2 >4$, arriving to $dy = pdx, ~dy = k^2 p dx$ where $I = k + \frac{1}{k}$.

When the associated Ricatti equation is elliptic or parabolic (namely, for $(M, \xi, \tilde\xi)_+$ with $I^2 \le 4$) we do not have a pair of constant solutions to \eqref{eq:Ric} at our disposal, and will need to make some more detailed computations in order to establish the stated normal form. We note that the character (hyperbolic or elliptic) of the Ricatti equation is related to the contact circles and contact hyperbolas of \cite{Geig}, \cite{Perrone}.

Let us take then locally a non-constant solution $\rho$ to \eqref{eq:Ric} with induced local co-frame (definition \ref{def:SchwCofr}) $dx, \eta, d\rho$ with structure equation eq.~\eqref{eq:StrEqsInt} reading (for $I = cst.$) as:
\[ d\eta = hdx\wedge \eta. \]
Note that if $h \equiv 0$, then we may take locally $dx = \alpha + \rho\tilde\alpha, dy =  \eta = (\rho + C)\alpha - (C\rho + 1)\tilde\alpha$ where $C = I/2$ is constant, and solving for the contact forms we arrive to the desired form of $\alpha\sim dy - pdx, ~\tilde\alpha \sim dy = \frac{ap + b}{cp+d}dx$ with $a,b,c,d$ constants. Thus, for this $(M, \xi, \tilde\xi)_+$ case with $I^2 \le 4$, we are reduced to showing that there exists locally some non-constant solution of \eqref{eq:Ric} with associated structure equation \eqref{eq:StrEqsInt} having $h \equiv 0$ in order to establish our stated normal form. The existence of such a solution will follow from:

\begin{lem}
    For $I\equiv cst.$ and $(M, \xi, \tilde\xi)_+$, let $\rho$ be some local non-constant solution to \eqref{eq:Ric}, with induced local co-frame  $dx,  \eta, d\rho$ and structure equation $d\eta = h dx\wedge\eta$, where $h = - \mu_x/\mu$, for $\mu\eta = dy$. Then, for a general solution of \eqref{eq:Ric}
    \[ \tilde\rho = \frac{\rho + c(1+C\rho)}{1 - c(\rho + C) } \]
    where $c(x,y)$ is an arbitrary function with $c' \equiv 0$, and $C= I/2$ is constant, having induced local co-frame 
    \[ d\tilde x, \tilde \eta, d\tilde\rho\]
    we have its structure equation $d\tilde\eta = \tilde h d\tilde x\wedge \tilde\eta$ where:
    \[ \tilde h = \frac{M_x + hM + \mu (C^2 - 1)(cM)_y}{M^2 (1 + c(1-C^2))} \] 
    where $M(x,y)$ satisfies: 
    \[ M_y + \left( \frac{cM}{\mu } \right)_x = 0.\]
\end{lem}

\begin{proof}
    We have: $\alpha + \tilde\rho \tilde\alpha = \frac{dx - c\eta}{1 - c(\rho +C)}$, so that $\theta + \tilde\rho \tilde \theta = d\tilde x$ for $(\theta,\tilde \theta) = \Lambda(\alpha, \tilde\alpha)$, and $\Lambda = M ( 1 - c(\rho + C))$, and $M(x,y)$ satisfying: $M_y + \left( \frac{cM}{\mu } \right)_x = 0$ where $\mu\eta = dy$. Then, directly differentiating $\frac{\tilde\eta}{M} = c(1 - C^2)dx + \eta$, we find the stated expression for $\tilde h$.
\end{proof}

From the Lemma, our result follows, since given $\rho$ with its resulting $h(x,y)$ and $\mu(x,y)$ through $\mu_x = - h\mu$, we may take $M(x,y), c(x,y)$ solutions of the system of quasi-linear 1st order PDE's: $M_y = - \left( \frac{cM}{\mu } \right)_x$ and $c_y = \frac{Mh + M_x}{\mu M(1-C^2)} + \frac{c}{M}\left(\frac{cM}{\mu}\right)_x$, which are in Cauchy-Kovalevskaya form, so admit (local) solutions (if $C = 1$, one can take $M = \mu$ and $c(x,y)$ solving $c_x = - \mu_y$). The resulting solution $\tilde\rho = \frac{\rho + c(1+C\rho)}{1 - c(\rho + C) }$ then has induced local co-frame with $\tilde h = 0$  so that $d\tilde x = \alpha + \tilde\rho \tilde\alpha, d\tilde y = \tilde\eta$ as desired.
\end{proof}

Otherwise, when $I$ is not constant, we have an additional invariant:
\begin{dfn}\label{def:Jinv}
    {\em For $(M, \xi, \tilde \xi)$ with $\mathcal{S} \equiv 0$ and $dI\ne 0$, there is an induced function:
    \[ \mathcal{J}: M \to \mathbb{RP}^1, ~~~\mathcal{J} = (j_1: j_2) \]
    by writing $dI = j_1\alpha + j_2 \tilde\alpha$, for some balanced contact forms. When $j_1\ne 0$, we set: $J = j_2/j_1$.}
\end{dfn}

For this case, $\mathcal{S}\equiv 0, dI\ne 0$, we have then always $j_1\ne 0$ or $j_2\ne 0$ and will without loss of generality consider that $j_1\ne 0$, up to a re-labeling of the pairs of contact structures. Then, we find the more refined form of proposition \ref{prop:genS0}:

\begin{prop}\label{prop:S0Gen}
    For $(M, \xi, \tilde \xi)_\pm$ with $\mathcal{S} \equiv 0$ and $dI\ne 0$ then
    \begin{enumerate}
        \item in case $J' \ne 0$, there are local coordinates $(x, I, J)$ in which:
        \[ \xi = \{ dI = BJ dx\}, ~~~\tilde \xi = \{ dI \pm \frac{Bdx}{I+J} = 0\} \]
        for some function $B(I,x) > 0$. The pair admits exactly one  symmetry iff $B$ is separable: $\del_I\del_x \log B \equiv 0$ (ie, its invariant $G$, defined below, \eqref{eq:G}, vanishes).

        \item in case $J' \equiv 0$ there are local coordinates $(x, I, p)$ in which: 
        \[ \xi = \{ dI = p dx \}, ~~\tilde \xi = \{ \pm J^2 dI =  p dx \} \]
        where $J(I)$ is a root of $1 + IJ \pm J^2 = 0$.

    \end{enumerate}
    
\end{prop}
\begin{proof}
    We will consider the commonly oriented case $(M, \xi, \tilde \xi)_+$ (the oppositely oriented case is handled in the same way), and scale our balanced contact forms so that $dI = \alpha + J\tilde\alpha$. Then $J$ is a solution to \eqref{eq:Ric}: $J' = 1 + IJ + J^2$.
In case $J'\ne 0$, we then have a definite co-frame on $M$:
\[ dI, \eta, dJ \]
where we take $\eta = (J + \frac{I}{2})( \alpha + J_* \tilde\alpha)$
as the integrable distribution associated to the solution: $J_* = - \frac{CJ + 1}{J+C}$ for $C = I/2$ of \eqref{eq:Ric}. Since we have in this case a co-frame on $M$ and at least two non-trivial invariants ($I,J$), we have by remark \ref{rem:counting} that {\em such pairs admit at most a one-dimensional local symmetry group}.  The structure equations \eqref{eq:StrEqsInt}, with $x = I$ and $\rho = J$ are then:
\[ d\eta = H dI \wedge\eta \]
where the coefficient $H$, having $H' = 0$, is another invariant, as well as its derivatives along this frame:
\begin{equation}\label{eq:G}
    dH = F dI + G \eta
\end{equation} 
being further invariants. If $G \ne 0$, then $H, I, J$ form local coordinates in invariants and such a pair admits {\em no} infinitesimal symmetries.
In any case, since $\eta$ is integrable, we may always choose locally some function $I_*$ with
\[ dI_* = \mu\eta,\]
by taking some integrating factor $\mu(I, I_*)$ satisfying $\del_I\mu = - H\mu$ (use $d\eta = HdI\wedge \eta$ and $H' = 0$). The contact distributions in these coordinates $(I_*, I, J)$ are then given by:
\[ dI_* + \mu (C + \frac{1}{J}) dI = 0, ~~\mu(C+J)dI = dI_*, ~~~~C = I/2. \]
Which we may put it in the stated normal form of the 1st item by taking another integrating factor $M( dI_* + \mu C dI) = - dx$
and setting $B = M\mu$. 
Note that if the pair is to admit exactly one symmetry (see remark \ref{rem:counting}), we must have $G = 0$, so that $dH = FdI$ and $H(I)$ is a function only of $I$, and then as well our integrating factor $\mu(I)$ (solving $\del_I\mu = - H(I)\mu$) can be taken as a function only of $I$ and the coefficient $B$ above is separable (we may take $B(I) = \mu(I)$). 

 Otherwise, when $J' = 1 + IJ \pm J^2 \equiv 0$, then $J(I)$ is a function of $I$. In the oppositely oriented case, we then have two integrable distributions $dI = \alpha + J\tilde\alpha$ and $dx = \mu ( J\alpha - \tilde\alpha)$, whereas in the commonly oriented case, necessarily  $I^2 > 4$ in order to have $J'\equiv 0$ and $dI\ne 0$, and we have the two integrable distributions $dI = \alpha + J\alpha$ and $dx = \mu ( J\alpha + \tilde\alpha)$. Solving for the contact distributions in either case yields the stated normal forms of the 2nd item when $J' \equiv 0$.
\end{proof}

\begin{prop}[Explicit formulas]\label{rem:S0Sum}
    In local coordinates $\xi = \{ dy = pdx\}, \tilde \xi = \{ dy = f(x,y,p) dx\}$, these invariants can be given more explicitly as follows. First recall proposition \ref{rem:explFormulasGen} where we have given $X, I$ explicitly in such coordinates, as well as a pair of balanced contact forms, $\alpha,\tilde\alpha$ in eq.~\eqref{eq:BalCoords}. Set:
    \begin{equation}\label{eq:BalCoordFrame}
        \mathcal{D} := \frac{\del_x + f\del_y}{f-p}, ~~\mathcal{\tilde D} := \frac{|f_p|^{1/2}}{|p-f|} (\del_x + p\del_y).
    \end{equation}
     When $\mathcal{S}\equiv 0$ and $dI\ne 0$, the invariant $J$ is given by:
    \[ J = \frac{\mathcal{\tilde D} \cdot I}{\mathcal{D}\cdot I}\]
    with $J' = X\cdot J = 1 + IJ\pm J^2$. When $J'\ne 0$, set 
     \[ X_1 := \frac{ J'\mathcal{D} - (\mathcal{D}\cdot J) X}{J' \mathcal{D}\cdot I}, ~~X_2 := \frac{J'\mathcal{\tilde D} - (\mathcal{\tilde D}\cdot J) X}{J' \mathcal{D}\cdot I} \]
     and the invariant $H$ is given by:
    \[ H = - (\mathcal{D}\cdot I) \tilde\alpha ( [X_1, X_2]) + \frac{J}{J'}. \]
    Setting $X_jH = H_j$, the invariants $F, G$ (the condition to admit --exactly one-- symmetry being $G\equiv 0$) are given by:
    \[ J' F = \left( \frac{IJ}{2} + 1\right) H_1 \pm \left( J \pm \frac{I}{2}\right) H_2, ~~~ \pm J' G = J H_1 - H_2.  \]
\end{prop}

\begin{proof}
    The frame $\mathcal{D}, \mathcal{\tilde D}, X$ is dual to the balanced co-frame $\alpha, \tilde\alpha, ds$ in \eqref{eq:BalCoords}, so that $(j_1 : j_2) = (\mathcal{D}\cdot I: \mathcal{\tilde D}\cdot I)$, and when $J'\ne 0$, then $X_1, X_2, X$ are dual to the co-frame $\theta, \tilde\theta, dJ/J'$ with $dI = \theta + J\tilde\theta$, and we obtain the stated equations from \eqref{eq:StrEqsRho}, \eqref{eq:StrEqsInt}, \eqref{eq:G} for this co-frame.
\end{proof}

\subsubsection{Associated path geometries}\label{rem:S0PG}
   A general pair $(M, \xi, \tilde \xi)_\pm$ with $\mathcal{S} \equiv 0$ and $dI\ne 0, J'\ne 0$ has associated to it a pair of {\em path geometries} (see eg ch.~1, \S 6 of \cite{Arnold}, or \cite{BGH}), where:

   \begin{dfn}
       {\em The geometric structure of a 2nd order ODE or a {\em path geometry} is a 3-manifold $M$ with a pair of line fields $L_1, L_2$ on $M$ spanning a contact distribution $\xi = L_1\oplus L_2$ on $M$.}
   \end{dfn}

   Namely, to a pair $(M, \xi, \tilde \xi)_\pm$ with $\mathcal{S} \equiv 0$ and $dI\ne 0, J'\ne 0$, we associate its {\em induced pair of path geometries}:
   \[ \xi = A\oplus L, ~~\tilde \xi = A\oplus \tilde L \]
   for $L = \xi \cap (\ker dJ), \tilde L = \tilde \xi\cap (\ker dJ)$. The projections of integral curves of $L, \tilde L$ under $M\to M/A$ each locally form a pair of 2-parameter families of curves in $M/A$. In the coordinates of proposition~\ref{prop:S0Gen}, these curves are given as graphs ($(x, I(x))$ or $(x(I), I)$ respectively) of solutions to the 2nd order ode's:
   \begin{equation}\label{eq:PG1}
       \frac{d^2I}{dx^2} = \left(\frac{B_x + B_I\frac{dI}{dx}}{B}\right) \frac{dI}{dx}, ~~~\frac{d^2x}{dI^2} = \mp \frac{1}{B} - \left(\frac{B_I + B_x \frac{dx}{dI}}{B}\right) \frac{dx}{dI}
   \end{equation}
   for a function $B(I, x) > 0$.
   
   \begin{rem}
       {\em Path geometries $(M, \xi = L_1\oplus L_2)$ have two basic (generating) invariants, say $I_1, I_2$. Let us recall here their following properties. The projections of integral curves of $L_1$ under $M\to M/L_2$ are unparametrized geodesics of some affine connection $\nabla_2$ on $M/L_2$ iff $I_2\equiv 0$ (likewise projections of integral curves of $L_2$ under $M\to M/L_1$ are unparametrized geodesics of some affine connection $\nabla_1$ on $M/L_1$ iff $I_1\equiv 0$). The condition $I_2\equiv 0$ on $M/L_2$ is simple to state in local coordinates $(x,y)$ on $M/L_2$: it is equivalent to this family of curves (the projections of integral curves of $L_1$) being given locally by the graphs $(x, y(x))$ of solutions to a 2nd order ode which is at most {\em cubic} in the 1st derivative, ie of the form: $y'' = \alpha(x,y) + \beta(x,y) y' + \gamma(x,y)(y')^2 + \delta(x,y)(y')^3$. Moreover, a relevant fundamental theorem of path geometry is that a path geometry has both $I_1\equiv 0$ and $I_2\equiv 0$ iff it is locally {\em flat}: there are some local coordinates $(x,y)$ on $M/L_2$ such that the projections of integral curves of $L_1$ under $M\to M/L_2$ are given by straight lines in the $xy$-plane. See for example eq.~(5) of \cite{BGH} for an explicit formula for the condition $I_1\equiv 0$ in terms of the 2nd order ODE $y'' = \omega(x,y,y')$ in local coordinates $(x,y)$ on $M/L_2$.}
   \end{rem}

   Since both 2nd order ODE's of \eqref{eq:PG1} are at most quadratic in the 1st derivative, their solutions are unparametrized geodesics of some affine connection. One may check, using for example eq.~(5) of \cite{BGH}) that the pair $(M, \xi, \tilde\xi)$ with $\mathcal{S}\equiv 0, dI\ne 0, J'\ne 0$ admits a symmetry ($B(I,x)$ of proposition \ref{prop:S0Gen} is separable: having $\del_I\del_x \log  B \equiv 0$) exactly when the 2nd order ODE's of \eqref{eq:PG1} are each flat (each equivalent to $y'' = 0$ under some change of coordinates). Curiously, there is no system of coordinates in which the general solution to both 2nd order ODE's \eqref{eq:PG1} are simultaneously rectified to straight lines.

\subsection{Non-zero Schwartzian} Now we consider when the induced projective structures on the axis are distinct ($\mathcal{S}\ne 0$). In this case, we have two invariants:
\[ I, I' \ne 0 \]
and will consider separately when these invariants are dependent $dI\wedge dI'\equiv 0$ and independent $dI\wedge dI'\ne 0$. First, for the dependent case, we have the general normal form:

\begin{prop}\label{prop:DepGenlnf}
    For $(M, \xi, \tilde \xi)$ with $\mathcal{S}\ne 0$ and $I'(I)$ a function of $I$, there are local coordinates in which: $\xi = \{ dy = p e^{c(x,y)}dx \}, ~\tilde \xi = \{ dy = f(p) e^{c(x,y)}dx\}$,
    for some functions $c(x,y)$ and $f(p)$.
\end{prop}

\begin{proof} The derivation is essentially the same as that in proposition~\ref{prop:genS0}. 
 We consider two solutions $\rho(I), \tilde\rho(I)$ to the Ricatti equation \eqref{eq:Ric}:
 \[ I'(I) \frac{d\rho}{dI} = 1 + I\rho \pm \rho^2 \]
 and scale some balanced contact forms to have $dx = \alpha + \rho \tilde\alpha$ and $dy = \mu( \alpha + \tilde\rho \tilde\alpha)$. The contact forms $\alpha, \tilde\alpha$ satisfy the structure equations \eqref{eq:StrEqsRho} with $d\beta = 0$, for $\beta = dI/I'$, whose coefficient $h$ has, by differentiating \eqref{eq:StrEqsRho} and using $d\beta = 0$ that: $h' = 0$, so that $h(x,y)$. The integrating factor $\mu$ for $\alpha + \tilde\rho\tilde\alpha$ then satisfies:
 \[ \frac{d\mu}{\mu} + hdx \pm (\tilde\rho - \rho)\beta \equiv 0 \mod dy \]
 and so has the form $\mu = e^{\fe(I) + c(x,y)}$ for $\fe(I) = \mp \int \frac{\tilde\rho(I) - \rho(I)}{I'(I)} dI$ (and $c_x = - h$). Solving for the contact distributions we have $\alpha \sim dy - \frac{\tilde\rho}{\rho}e^\fe e^c dx, ~\tilde\alpha \sim dy - e^\fe e^c dx$, and the stated form above by setting $p = \frac{\tilde\rho(I)}{\rho(I)}e^{\fe(I)}$, $e^{\fe(I)} = f(p)$.
\end{proof}

In case there are symmetries:
\begin{prop}\label{prop:SymmsDepGen}
    If a pair $(M, \xi, \tilde \xi)$ with $\mathcal{S}\ne 0$ and $I'(I)\ne 0$ a function of $I$ admits an infinitesimal symmetry then either
    \begin{enumerate}
        \item there are local coordinates $(x,y,p)$ in which:
        \[ \xi = \{ dy = p dx\}, ~~\tilde \xi = \{ dy = f(p) dx\}, \]
        \item or, there are local coordinates $(x,y,p)$ in which:
    \[ \xi = \{  dy = \frac{p + a(y)}{p + b(y)} dx \} , ~~\tilde \xi = \{ dy = \frac{f(p) + a(y)}{f(p) + b(y)} dx \}. \]
    \end{enumerate}
\end{prop}
\begin{proof}
    We prolong in this case to consider the principal $\R^\times$ sub-bundle of the co-frame bundle on $M$
    \[ \R^\times \to B\to M \]
    consisting of co-frames $\alpha, \tilde\alpha, \beta = dI/I'$ on $M$, where the $\R^\times$ action is by simultaneous re-scalings $\alpha,\tilde\alpha \mapsto \lambda\alpha, \lambda\tilde\alpha$ on the balanced contact forms $\alpha, \tilde\alpha$. We first claim that we have on $B$ a canonical co-frame induced by the pair $(M,\xi, \tilde \xi)$ with $I'(I)\ne 0$. Let $\Theta, \tilde\Theta, \beta$ be the tautological 1-forms on $B$: $\Theta_b(v) = \alpha(\pi_*v), \tilde\Theta_b(v) = \tilde\alpha(\pi_*v)$, and $\beta = \pi^*\beta$ for $b = (\alpha, \tilde\alpha, \beta)\in B$ and $\pi:B\to M$. Then:
    \begin{lem}
         There is a unique connection 1-form $\nu$ on $B$ for which the co-frame $\Theta, \tilde\Theta, \beta, \nu$ satisfies the structure equations:
    \begin{equation}\label{eq:liftStrEq}
        d\Theta = \nu\wedge\Theta + \tilde\Theta\wedge \beta, ~~~d\tilde\Theta = \nu\wedge\Tilde\Theta + I\tilde\Theta\wedge\beta \pm \beta\wedge\Theta , ~~d\beta = 0, ~~d\nu = \Lambda \Theta\wedge\tilde\Theta,
    \end{equation}
    the coefficient $\Lambda:B\to \R$ being another invariant of $(M,\xi, \tilde \xi)$.
    \end{lem}
    \begin{proof}
        Indeed, for any choice of balanced contact forms with trivialization $B\cong M\times\R\ni (m, \lambda)$ through $\Theta = \lambda\alpha, \tilde\Theta = \lambda \tilde\alpha$, we take:
    \[ \nu = \frac{d\lambda}{\lambda} + h \alpha - \tilde h \tilde\alpha + g\beta,   \]
    where $\alpha, \tilde\alpha, \beta$ satisfy \eqref{eq:GenBalStrEq}.
    \end{proof}
    Now, having a canonical co-frame on the 4-dimensional $B$ with (at-least) one non-trivial invariant $I$, such a pair can admit at most a 3 dimensional symmetry algebra: only when $\Lambda(I)$ is dependent on $I$. It follows directly by differentiating \eqref{eq:liftStrEq} that:
    \[ d\Lambda = \Lambda_1 \Theta + \Lambda_2\tilde\Theta + I\Lambda\beta - 2\Lambda\nu \]
    so that $\Lambda$ is dependent on $I$ ( $I'(I)\beta = dI$ ) iff $\Lambda\equiv 0$. In case $\Lambda\equiv 0$, we may take a co-frame as in proposition~\ref{prop:DepGenlnf}, and directly compute that $\Lambda\equiv 0$ is equivalent to $\del_x\del_y c \equiv 0$, ie $c(x,y) = a(x) + b(y)$ is separable, and the normal form of proposition~\ref{prop:DepGenlnf} takes the form stated in item 1 above.

    Otherwise, when $\Lambda\ne 0$, we have two independent and non-trivial invariants, $I, \Lambda$, and (at most) a 2-dimensional symmetry algebra. As it turns out:
    \begin{lem}
        A pair $(M, \xi, \tilde\xi)$ with $I'(I)\ne 0$ and $\Lambda\ne 0$ admits at most a 1-dimensional symmetry algebra.
    \end{lem}
    \begin{proof}
        To realize a 2-dimensional symmetry algebra the additional invariants $\Lambda_1, \Lambda_2$ would need to be functionally dependent on $I, \Lambda$.
    One computes directly that:
    \[ d\Lambda_1 = \Lambda_{11}\Theta + \Lambda_{12}\tilde\Theta + (I\Lambda_1 \mp \Lambda_2) \beta - 3\Lambda_1\nu, ~~ d\Lambda_2 = \Lambda_{21}\Theta + \Lambda_{22} \tilde\Theta + (\Lambda_1 + 2I\Lambda_2)\beta - 3\Lambda_2\nu \]
where $\Lambda_{21} - \Lambda_{12} = 2 \Lambda^2$.
In particular, if $d\Lambda_1, d\Lambda_2 \in \langle dI, d\Lambda\rangle$ (and $\Lambda\ne 0$) then we find the contradiction:
\[ 2\Lambda^3 = \Lambda(\Lambda_{21} - \Lambda_{12}) = 3( \Lambda_1\Lambda_2 - \Lambda_1\Lambda_2) = 0. \]
    \end{proof}

     When $\Lambda\ne 0$, we may pass back down to $M$, where we have the co-frame:
\[ \theta = |\Lambda|^{1/2}\Theta, ~\tilde\theta = |\Lambda|^{1/2} \tilde\Theta, ~\beta = dI/I' \]
whose structure equations we compute as:
\begin{equation}\label{eq:DownEstr}
 d\theta = - \frac{n}{2}\theta\wedge\tilde\theta + \tilde\theta\wedge\beta + \frac{I}{2} \beta\wedge\theta, ~~d\tilde\theta = \frac{m}{2} \theta\wedge\tilde\theta + \frac{I}{2}\tilde\theta\wedge\beta \pm \beta \wedge \theta,~~
d\beta = 0
\end{equation}
for the invariants
\[ m = \Lambda_1/|\Lambda|^{3/2}, ~n = \Lambda_2/|\Lambda|^{3/2} \]
on $M$, having
\[ dm = m_1\theta + m_2\tilde\theta -  \left( \frac{Im}{2} \pm n\right)\beta, ~~ dn = n_1 \theta + n_2 \tilde\theta + \left( m + \frac{In}{2} \right)\beta,  \]
\[n_1 = m_2 + 2. \]
In general $I, m,n:M\to \R$ are independent, and such a pair with $\Lambda\ne 0$ admits no infinitesimal symmetries.  Such a pair admits an infinitesimal symmetry only when the differentials $\langle dm, dn, dm_1, dm_2, dI\rangle$ have rank two:
\begin{equation}\label{eq:Dep1SymCond}
    0 \equiv \begin{vmatrix}
    m_1 & m_2 \\ n_1 & n_2
\end{vmatrix},  ~~0 \equiv \begin{vmatrix}
    m_1 & m_2 \\ n_{21} & n_{22}
\end{vmatrix}, ~~0 \equiv \begin{vmatrix}
    m_1 & m_2 \\ m_{21} & m_{22}
\end{vmatrix}, ~~0 \equiv \begin{vmatrix}
    m_1 & m_2 \\ m_{11} & m_{12}
\end{vmatrix}.
\end{equation}
Supposing the condition \eqref{eq:Dep1SymCond} holds, we now proceed to determine the normal form of item 2, by seeking appropriate solutions of our Ricatti equation \eqref{eq:Ric}. Note that for any function $f:M\to \R$, writing $df = f_1\theta + f_2\tilde\theta + f'\beta$, then by differentiating and using \eqref{eq:DownEstr}, one obtains the relations:
\begin{equation}\label{eq:Dep1SymDerivs}
    (f_2)' - (f')_2 = f_1 + \frac{I}{2} f_2, ~~ (f')_1 - (f_1)'  =  \frac{I}{2}f_1 \pm f_2 , ~~ f_{21} - f_{12} = \frac{nf_1 - mf_2}{2}
\end{equation} 
for its higher derivatives.
From \eqref{eq:Dep1SymCond}, we set:
\[ \rho = m_2/m_1 = n_2/n_1, ~~\ell = n/m,  \]
and check from \eqref{eq:Dep1SymDerivs} that $\rho, \ell$ are solutions to \eqref{eq:Ric} and, for $d\rho = \rho_1\theta + \rho \rho_1\tilde\theta + \rho'\beta$, we have:
\[ \ell = \rho + 2\frac{\rho_1}{m}. \]
Note that we cannot have $\rho_1\equiv 0$ (ie $\ell\equiv\rho$) since differentiating $n = \rho m$, and using $n_1 = m_2 + 2$ one would then have:
\[ m_2 = \rho m_1 = n_1 = m_2 + 2, \]
which is impossible, so that $\rho_1\not\equiv 0$, and $\ell\not\equiv\rho$ are distinct solutions to \eqref{eq:Ric}. Then, using \eqref{eq:DownEstr} -- \eqref{eq:Dep1SymDerivs}, one computes:
\[ d\left( \frac{\theta + \rho\tilde\theta}{\rho_1} \right) = 0 , ~~~ d\left( m (\theta + \ell\tilde\theta) \right) = dy \wedge \left( m(\theta + \ell\tilde\theta) \right)  \]
so that for the integrable distributions $\theta + \rho\tilde\theta, \theta + \ell \tilde\theta$ we have integrating factors:
\[ \rho_1dy = \theta + \rho\tilde\theta, ~~ \frac{e^y}{m} dx = \theta + \ell\tilde\theta, \]
and the contact forms are given through:
\begin{equation}\label{eq:ContFormsDep}
    m(\ell - \rho) \theta = \ell m\rho_1 dy - \rho e^{y} dx, ~~m(\rho - \ell)\tilde\theta = m\rho_1dy - e^y dx.
\end{equation}
To give them the more explicit form of item 2 above, we note, from \eqref{eq:Dep1SymCond}, that all the invariants depend only on $y, I$, and for a function $f(y,I)$ we have:
\[ \del_yf = \rho_1f_1. \]
In particular, we compute that $(\rho_1\ell_1 m^2)' \equiv 0$, so that:
\[ m^2 \del_y\ell  = \rho_1\ell_1 m^2 = c(y) \]
is some function of $y$. Now, since $\rho, \ell$ are solutions to \eqref{eq:Ric}, they have the form:
\[ \frac{\ell - r_1}{\ell - r_2} = a(y) e^\lambda, ~~\frac{\rho - r_1}{\rho - r_2} = b(y) e^\lambda \]
where $r_j(I)$ are any pair of solutions to \eqref{eq:Ric} depending only on $I$ (and $\lambda(I) = \pm \int \frac{r_1 - r_2}{I'} dI$). Computing then the partials $\del_y\ell = c(y)/m^2, \del_y\rho = (\rho_1)^2$ and substituting into \eqref{eq:ContFormsDep} we arrive to the contact 1-forms:
\[ dY = \gamma(y) dy = \frac{r_1(I) - a(Y) e^{\lambda(I)}r_2(I) }{r_1(I) - b(Y) e^{\lambda(I)}r_2(I)} dx, ~~dY = \frac{1 - e^{\lambda(I)}a(Y)}{1 - e^{\lambda(I)}b(Y)}dx \]
for some function $\gamma(y)$,
which is in the stated form of item 2 above (set $P = \frac{r_1}{e^\lambda r_2}$ and $f(P) = e^{-\lambda}$).
\end{proof}

\begin{rem}[associated path geometry]\label{rem:DepPG}
    {\em Similarly to our observations in section \ref{rem:S0PG}, we have for a pair $(M, \xi, \tilde \xi)$ with $\mathcal{S}\ne 0$ and $I'(I)\ne 0$ an associated pair of path geometries on $M/A$ by projecting integral curves of $\xi\cap \ker (dI), \tilde \xi\cap \ker(dI)$ under $M\to M/A$. In the coordinates of proposition~\ref{prop:DepGenlnf} we see such integral curves project always to the {\em same} 2-parameter family of curves in $M/A$, given in these coordinates as graphs, $(x, y(x))$, of solutions to the 2nd order ode:
    \[ \frac{d^2y}{dx^2} = \left( c_x + c_y \frac{dy}{dx} \right)\frac{dy}{dx}. \]
    They are unparametrized geodesics of some affine connection on $M/A$, and we check (again with eq.~(5) from \cite{BGH}) that such a pair admits the maximal 3-dimensional symmetry algebra ($\del_x\del_y c \equiv 0$) iff this 2nd order ODE is {\em flat} (locally equivalent to $Y'' = 0$).}
\end{rem}

Next, we consider the general case when we have two independent invariants: $dI\wedge dI'\ne 0$. In this case we have further invariants:
\begin{dfn}\label{def:EStrInvs}
    {\em For $(M, \xi, \tilde \xi)$ with $\mathcal{S}\ne 0$ and $dI\wedge dI'\ne 0$, there is an induced function:
    \[ \mathcal{K}: M\to \mathbb{RP}^1, ~~\mathcal{K} = (k_1: k_2)\]
    by writing $dI' = k_1\alpha + k_2\tilde\alpha + \frac{I''}{I'}dI$ for some balanced contact forms. When say $k_1\ne 0$, we set $K:= k_2/k_1$. Moreover, we have then a co-frame $\alpha, \tilde\alpha, \beta = dI/I'$ determined on $M$, where the balanced contact forms are determined by the scaling condition:
    \[ dI' = \alpha + K\tilde\alpha + I''\beta, ~~~~(\beta = dI/I'). \]
    The structure equations of this co-frame are given in \eqref{eq:EstrEqs} below.}
\end{dfn}

So, in this situation we have a co-frame induced on $M$, as well as (at least) two non-trivial independent invariants: $I, I'$. There can be at most one non-trivial infinitesimal symmetry of such a pair. Its normal form, by proposition~\ref{prop:genLNF}, is then given by item IV of theorem \ref{thm:lnf}, and to finish, we determine the conditions on its additional invariants for when this is the case:

\begin{prop}\label{prop:GenSymm}
    A pair $(M, \xi, \tilde \xi)_\pm$ with $\mathcal{S}\ne 0$ and $dI\wedge dI'\ne 0$ admits an infinitesimal symmetry iff:
    \[ K' = 1 + IK \pm K^2, \]
    \[ K K_1 = K_2, ~~ K H_1 = H_2 \]
    where $F_j := X_j F, ~F' = XF$ are the derivatives along the frame $X_1, X_2, X$ dual to the co-frame $\alpha, \tilde\alpha, \beta$ of definition~\ref{def:EStrInvs}, and the invariant $H$ is given by the structure equations of this co-frame \eqref{eq:EstrEqs} below. 
\end{prop}
\begin{proof}
    From \eqref{eq:GenBalStrEq}, and $\beta = dI/I'$, we have for the co-frame of definition~\ref{def:EStrInvs} the structure equations:
    \begin{equation}\label{eq:EstrEqs}
        d\alpha = \tilde H \alpha\wedge\tilde\alpha + \tilde\alpha\wedge\beta + G \beta\wedge\alpha , ~~~d\tilde\alpha = H \alpha\wedge\tilde\alpha + (I-G)\tilde\alpha\wedge\beta \pm \beta\wedge\alpha 
    \end{equation}
    \[ d\beta = - \frac{dI'}{I'}\wedge\beta = \frac{-K\tilde\alpha\wedge\beta + \beta\wedge\alpha}{I'} \]
    whose coefficients, upon differentiating $dI' = \alpha + K \tilde\alpha + I''\beta$, satisfy the relations:
    \[ \tilde H = - KH -K_1, ~~G = L_1 \mp  K - \frac{I''}{I'},  \]
    \[ KL_1 - L_2 = 1 + KI \pm  K^2 - K' \]
    for: $dI'' = L_1 \alpha + L_2 \tilde\alpha + I'''\beta$ and $dK = K_1 \alpha + K_2 \tilde\alpha + K'\beta$. The coefficient $H$ in these structure equations is an additional invariant, and we can consider as generating invariants:
    \[ H, I \]
    from which all further invariants are determined from successive derivatives of $H,I$ along the frame.
    
    Assuming such a pair admits an infinitesimal symmetry, say the vector field $Y$, it must preserve all these invariants (which must all be dependent upon $I, I'$). For $X_1, X_2, X$ the dual frame to $\alpha, \tilde\alpha, \beta$, then to preserve $I, I'$ we have the condition $Y\sim KX_1 - X_2$. The distribution spanned by the axis $A$ and $Y$ is an integrable distribution (since any symmetry $Y$ preserves the axis) and this integrable distribution is given by the kernel of $\alpha + K\tilde\alpha$, ie $K$ must be a solution to \eqref{eq:Ric}. To determine the complete conditions, note for a general function $F$ with $dF = F_1\alpha + F_2 \tilde\alpha + F'\beta$, we have
    \[ F_{21} - F_{12} + F_2H = F_1 (KH + K_1), \]
    \[ (F')_2 - (F_2)' + F_1 + \left(I \pm K + \frac{I''}{I'} - L_1\right)F_2 = K \frac{F'}{I'},\]
    \[ (F_1)' - (F')_1 + \left( L_1 \mp K - \frac{I''}{I'} \right)F_1 \pm  F_2 + \frac{F'}{I'} = 0. \]
    For this pair to admit a symmetry, we must ensure that all invariants (and their successive derivatives) satisfy $KF_1 = F_2$ (ie $YF = 0$). For such a function, the previous formulas for its higher derivatives simplify to:
    \begin{equation}\label{eq:SucDer}
        F_{21} - F_{12} = F_1K_1, ~~~ (F')_2 = K(F')_1
    \end{equation}
    Using that $K$ satisfies \eqref{eq:Ric}. Now, we claim that if additionally $K K_1 = K_2$, then for any function $F$ satisfying $YF = 0$, the same will hold for its higher derivatives. Indeed, by the second of \eqref{eq:SucDer}  we have that $F'$ satisfies $Y(F') = 0$. Differentiating $KF_1 = F_2$ by $X_1$, and using the first of \eqref{eq:SucDer}, we find:
    \[ K_1F_1 + KF_{11} = F_{21} = F_{12} + F_1K_1 \so KF_{11} = F_{12} \]
    and so $F_1$ satisfies as well $YF_1 = 0$. Now we find that so too does $YF_2 = 0$, since:
   \[ F_{22} = K_2F_1 + K F_{12} = KK_1 F_1 + K F_{12} = K ( K_1F_1 + F_{12}) = KF_{21} \]
   from $F_{21} = (KF_1)_1 = K_1 F_1 + K F_{11} = K_1 F_1 + F_{12}$. Since all the invariants are generated from $H, I$ (and $I$ already satisfies $YI= 0$), imposing as well that $YH = 0$, ie $KH_1 = H_2$, along with $K$ satisfying \eqref{eq:Ric} and $KK_1 = K_2$, then all invariants, $F$ (derived from $H,I$), satisfy $YF = 0$. In particular, all invariants are then functionally dependent upon $I, I'$ as needed, and such a pair (see remark \ref{rem:counting}) admits exactly one infinitesimal symmetry.
\end{proof}

Writing the appropriate structure equations in terms of their dual frames, similarly to proposition \ref{rem:S0Sum}), we find the following formulas for these invariants:

\begin{prop}[Explicit formulas]\label{rem:GenSum}  In local coordinates with $\xi = \{ dy = pdx\}, \tilde \xi = \{ dy = f(x,y,p)dx\}$, we have given above explicit formulas for $I, X$ (proposition \ref{rem:explFormulasGen}) in these coordinates, balanced contact forms, $\alpha, \tilde\alpha$ eq.~\eqref{eq:BalCoords}, and dual vector fields $\mathcal{D}, \mathcal{\tilde D}$, eq.~\eqref{eq:BalCoordFrame}. 
In case $I'\ne 0$ and $dI\wedge dI'\equiv 0$, set
\[ Y = \mathcal{D}- \frac{\mathcal{D}\cdot I}{I'} X, ~~\tilde Y = \mathcal{\tilde D} - \frac{\mathcal{\tilde D}\cdot I}{I'}X\]
\[ h = - \tilde\alpha ([Y, \tilde Y]), ~~\tilde h = - \alpha([Y, \tilde Y]), ~~\lambda = \sqrt{|Y\tilde h + \tilde Yh|}. \]
The invariant $\Lambda$ from the proof of proposition \ref{prop:SymmsDepGen} being a non-zero multiple of $\lambda$. Such a pair admits a 3-dimensional symmetry algebra (item 1 of proposition~\ref{prop:SymmsDepGen}) iff $\lambda \equiv 0$. When $\lambda\ne 0$, we have additional invariants $m,n$ given by:
\[ m = -2\lambda \tilde\alpha([Y_1, Y_2]), ~~n = 2\lambda\alpha([Y_1, Y_2]) \]
for $Y_1 = Y/\lambda, Y_2 = \tilde Y/\lambda$. The conditions to admit (exactly one) symmetry being given above \eqref{eq:Dep1SymCond}  (with $f_j = Y_jf$). In case $I'\ne 0$ and $dI\wedge dI'\ne 0$, set:
    \[ X_1 = \frac{I' \mathcal{D} - (\mathcal{D}\cdot I) X}{I'\mathcal{D}\cdot I' - I'' \mathcal{D}\cdot I}, ~~X_2 = \frac{I' \mathcal{\tilde D} - (\mathcal{\tilde D}\cdot I) X}{I' \mathcal{D}\cdot I' - I'' \mathcal{D}\cdot I}. \]
    The invariants $K,H$ are then:   
    \[ K = \frac{I'\mathcal{\tilde D}\cdot I' - I'' \mathcal{\tilde D}\cdot I}{I'\mathcal{D}\cdot I' - I'' \mathcal{D}\cdot I}, ~~~ H = - (\mathcal{D}\cdot I' - \frac{I''}{I'}\mathcal{D}\cdot I) \tilde\alpha ( [X_1, X_2]). \]
    The conditions to admit (exactly one) symmetry being given above in proposition~\ref{prop:GenSymm}  (with $F_j = X_jF$).
\end{prop}

\subsection{proof of theorem \ref{thm:lnf}} \label{sec:prf} 

\begin{table}[h]
    \centering
    \begin{tabular}{c|c|c}
    \hline\hline
  I$_1$   & $\mathcal{S}\equiv 0, dI \equiv 0$   & $(M, D, \tilde D)_+$, $I^2 = cst. > 4$, or $(M, D, \tilde D)_-$, $I = cst.$ \\
  \hline 
  I$_2$  & $\mathcal{S}\equiv 0, dI \equiv 0$   & $(M, D, \tilde D)_+$, $I^2 = cst. \le 4$. \\
    \hline\hline 
   II$_1$ & $\mathcal{S}\equiv 0, dI\ne 0$  & $J'\ne 0, G\equiv 0$.   \\
   \hline 
   II$_2$ & $\mathcal{S}\equiv 0, dI\ne 0$  & $J'\equiv 0$   \\
   \hline\hline 
   III$_1$  &  $\mathcal{S}\ne 0, dI\wedge dI'\equiv 0$  &  $\Lambda \ne 0, ~\frac{m_1}{m_2} = \frac{n_1}{n_2} = \frac{m_{11}}{m_{12}} = \frac{m_{21}}{m_{22}} = \frac{n_{21}}{n_{22}}$  \\
   \hline 
   III$_2$  &  $\mathcal{S}\ne 0, dI\wedge dI'\equiv 0$  &  $\Lambda \equiv 0$    \\
     \hline\hline 
    IV &  $\mathcal{S}\ne 0, dI\wedge dI'\ne 0$  &  $K' = 1 + IK \pm K^2, K_2 = KK_1, H_2 = KH_1$.   \\
    \hline\hline 
    \end{tabular}
    \caption{The characterization of the normal forms listed in theorem \ref{thm:lnf} in terms of their invariants (see corollary \ref{rem:Schwartz}, definitions~\ref{def:BalGenInv}, \ref{def:Jinv}, \ref{def:EStrInvs} and propositions \ref{rem:explFormulasGen}, \ref{rem:S0Sum} and \ref{rem:GenSum} above). }
    \label{table:Summary}
\end{table}

The proof of the main theorem \ref{thm:lnf} follows upon gathering the results of \S \ref{sec:Comps} as follows (see table \ref{table:Summary} for a summary):
\begin{proof}(theorem \ref{thm:lnf}) If a pair $(M, \xi, \tilde \xi)$ admits an infinitesimal symmetry (transverse to its axis), then by proposition~\ref{prop:genLNF}, it has the general local normal form of item IV.  A typical pair $(M, \xi, \tilde \xi)_\pm$ has $\mathcal{S}\ne 0$ and $dI\wedge dI'\ne 0$, and so by proposition~\ref{prop:GenSymm}, it admits (exactly one) infinitesimal symmetry iff its invariants $K,H$ satisfy the conditions stated in the last row of table \ref{table:Summary}. In case a pair with $\mathcal{S}\ne 0$ has $dI\wedge dI'\equiv 0$, then by the proof of proposition~\ref{prop:SymmsDepGen}, it admits a symmetry iff either its invariant $\Lambda\equiv 0$, and it has local normal form of III$_2$ (admitting exactly a three dimensional infinitesimal symmetry algebra), or $\Lambda\ne 0$, and its additional invariants $m,n$ satisfy \eqref{eq:Dep1SymCond} (the conditions stated in the 5th row of table \ref{table:Summary}) having the local normal form of III$_1$ and admitting exactly one infinitesimal symmetry. In case the pair $(M, \xi, \tilde \xi)$ has $\mathcal{S}\equiv 0$ and $I\equiv cst.$ its normal form is given by proposition~\ref{prop:Icst}, which can be split into type I$_1$:  consisting of $(M, \xi, \tilde \xi)_+$ with $I^2 >4$ and $(M, \xi, \tilde \xi)_-$ with $I = cst.$ (the types with hyperbolic Ricatti equation \eqref{eq:Ric}), or the type I$_2$: $(M, \xi, \tilde \xi)_+$ with $I^2 \le 4$ (the types with an elliptic or parabolic Ricatti equation \eqref{eq:Ric}). In case we have $\mathcal{S}\equiv 0$ and $dI\ne 0$, then by proposition~\ref{prop:S0Gen}, such a pair admits a symmetry iff either its invariant $J$ has $J'\equiv 0$ and we have the normal form of II$_2$, or when $J'\ne 0$ then its additional invariant $G$ must vanish identically, and we have the normal form of type II$_1$ (admitting exactly one infinitesimal symmetry). One may determine the symmetry algebras stated in theorem \ref{thm:lnf} from \eqref{eq:symm}.
\end{proof}

\section*{Acknowledgments}

I would like to thank Francesco Ruscelli, Agustin Moreno, Surena Hozoori, Federico Salmoiraghi, Gil Bor, and Taylor Klotz for their interest and helpful discussions, as well as the anonymous referee for their helpful comments on references and relevant examples of projective Anosov flows. I appreciate the support of the Deutsche Forschungsgemeinschaft (DFG, German Research Foundation) – Project-ID 281071066 – TRR 191.


\begin{thebibliography}{99}


\bibitem{Arnold} V.I.~Arnold, \textit{Geometrical methods in the theory of ordinary differential equations}. Vol. 250. Springer Science \& Business Media, (2012).

\bibitem{DSIV} V.I.~Arnold, A.B.~Givental, {\em Symplectic geometry}, Dynamical Systems IV: Symplectic geometry and its applications. Vol. 4. Springer Science \& Business Media, (2001).


\bibitem{BGH} R.~Bryant, P.~Griffiths, L.~Hsu, \textit{Toward a geometry of differential equations}. Geometry, Topology, \& Physics, Conf. Proc. Lecture Notes Geom. Topology. Vol. 4. (1995).



\bibitem{ThurstEl} Y.~Eliashberg, W.~Thurston. {\em Confoliations}. Vol. 13. American Mathematical Soc., (1998).

\bibitem{Gard} R.~Gardner. {\em The method of equivalence and its applications}. Society for industrial and applied mathematics, (1989).

\bibitem{Geig} H.~Geiges, J.~Gonzalo. {\em Contact circles on 3-manifolds}. Journal of Differential Geometry 46.2 (1997): 236-286.

\bibitem{Ghys} E.~Ghys. {\em D\'eformations de flots d'Anosov et de groupes fuchsiens}. Annales de l'institut Fourier. Vol. 42. No. 1-2, (1992): 209-247.

\bibitem{Ghys2} E.~Ghys. {\em Rigidit\'e diff\'erentiable des groupes fuchsiens}. Publications Math\'ematiques de l'IHÉS 78 (1993): 163-185.

\bibitem{Hoz} S.~Hozoori, {\em Symplectic geometry of Anosov flows in dimension 3 and bi-contact topology}. Advances in Mathematics 450 (2024): 109764. 

\bibitem{KhTabContInt} B.~Khesin, S.~Tabachnikov, {\em Contact complete integrability}. Regular and Chaotic Dynamics 15: 504-520 (2010).

\bibitem{Wilk} T.~Klotz, G.~Wilkens. {\em On Some Local Geometry of Bi-Contact Structures.} arXiv preprint arXiv:2312.01360 (2023).


\bibitem{Mass} T.~Massoni, {\em Anosov flows and Liouville pairs in dimension three}. arXiv preprint arXiv:2211.11036 (2022).


\bibitem{Mits} Y.~Mitsumatsu, {\em Anosov flows and non-Stein symplectic manifolds}. Annales de l'institut Fourier. Vol. 45. No. 5. (1995): 1407-1421.

\bibitem{TabOvSch} V.~Ovsienko, S.~Tabachnikov,  {\em What is the Schwarzian derivative?}. Notices of the AMS 56.1: 34-36 (2009).

\bibitem{Perrone} D.~Perrone, {\em Taut contact hyperbolas on three-manifolds}. Annals of Global Analysis and Geometry 60.3 (2021): 735-765.










\end{thebibliography}
\end{document}